\documentclass[12pt]{article}

\usepackage{cite}
\usepackage{lineno}

\usepackage{ifthen}
\usepackage{makeidx}
\usepackage{amsmath,amssymb,amstext} 
\usepackage{graphicx} 
\usepackage{fullpage}
\usepackage{datetime}
\numberwithin{equation}{section}
\usepackage{enumitem}
\usepackage{multirow}
\numberwithin{figure}{section}

\usepackage{booktabs}

\usepackage{algorithm,algorithmic} 

\numberwithin{table}{section}
\numberwithin{algorithm}{section}
 
\usepackage[letterpaper=true,pagebackref=true]{hyperref} 

\hypersetup{
    plainpages=false,       
    pdfpagelabels=true,     
    bookmarks=true,         
    unicode=false,          
    pdftoolbar=true,        
    pdfmenubar=true,        
    pdffitwindow=false,     
    pdfstartview={FitH},    
    pdftitle={
Implicit redundancy and degeneracy in conic programs    },    
    pdfnewwindow=true,      
    colorlinks=true,        
    linkcolor=blue,         
    citecolor=green,        
    filecolor=magenta,      
    urlcolor=cyan           
}

\usepackage{comment}

\usepackage{float}
\usepackage{bbm}
\usepackage{exscale}
\usepackage{tabularx}
\usepackage{syntonly}
\usepackage{amssymb}
\usepackage{amsmath}
\usepackage{amsthm}
\usepackage{latexsym}
\usepackage{makeidx}
\usepackage{longtable}
\usepackage[noabbrev]{cleveref}
\numberwithin{equation}{section}

\makeindex
\usepackage{chngcntr}
\usepackage{bm}

\def\R{\mathbb{R}}
\def\Sc{\mathbb{S}}
\def\Sn{\Sc^n}

\def\Snp{\Sc_+^n}

\def\Rm{\mathbb{R}^m}

\def\Rn{\mathbb{R}^n}

\def\Rnp{\mathbb{R}_+^n}
\def\Rkp{\mathbb{R}_+^k}

\def\Kexp{K_{\text{exp}}}

\def\eqref#1{{\normalfont(\ref{#1})}}

\def\eqref#1{{\normalfont(\ref{#1})}}

\newtheorem{theorem}{Theorem}[section]

\newtheorem{definition}[theorem]{Definition}
\newtheorem{example}[theorem]{Example}

\newtheorem{prop}[theorem]{Proposition}
\newtheorem{proposition}[theorem]{Proposition}

\newtheorem{corollary}[theorem]{Corollary}

\newtheorem{remark}[theorem]{Remark}

\newtheorem{lemma}[theorem]{Lemma}

\crefname{thm}{Theorem}{Theorems}
\Crefname{thm}{Theorem}{Theorems}
\crefname{problem}{Problem}{Theorems}
\Crefname{problem}{Problem}{Theorems}
\crefname{conjecture}{Conjecture}{Theorems}
\Crefname{conjecture}{Conjecture}{Theorems}
\crefname{proposition}{Proposition}{Propositions}
\Crefname{proposition}{Proposition}{Propositions}
\crefname{prop}{Proposition}{Propositions}
\Crefname{prop}{Proposition}{Propositions}
\crefname{cor}{Corollary}{Corollaries}
\Crefname{cor}{Corollary}{Corollaries}
\crefname{lem}{Lemma}{Lemmas}
\Crefname{lem}{Lemma}{Lemmas}
\theoremstyle{definition}
\crefname{definition}{definition}{definitions}
\Crefname{definition}{Definition}{Definitions}
\crefname{defn}{definition}{definitions}
\Crefname{defn}{Definition}{Definitions}
\crefname{remark}{Remark}{Remarks}
\Crefname{remark}{Remark}{Remarks}
\crefname{rmk}{Remark}{Remarks}
\Crefname{rmk}{Remark}{Remarks}
\crefname{example}{Example}{Examples}
\Crefname{example}{Example}{Examples}
\crefname{align}{}{}
\Crefname{align}{}{}
\crefname{equation}{}{}
\Crefname{equation}{}{}

\newcommand{\textdef}[1]{\textit{#1}\index{#1}}

\newcommand{\<}{\langle}
\renewcommand{\>}{\rangle}

\newcommand{\bE}{\mathbb{E}}
\newcommand{\bS}{\mathbb{S}}

\newcommand{\cN}{{\mathcal N} }

\newcommand{\cR}{{\mathcal R} }

\newcommand{\cL}{{\mathcal L} }
\newcommand{\cX}{{\mathcal X} }
\newcommand{\cI}{{\mathcal I} }

\newcommand{\cF}{{\mathcal F} }

\newcommand{\cA}{{\mathcal A} }
\newcommand{\cK}{{\mathcal K} }

\newcommand{\cT}{{\mathcal T} }

\newcommand{\COPn}{\text{COP}^n}

\newcommand{\FR}{\textbf{FR}\,}
\newcommand{\FRp}{\textbf{FR}}

\newcommand{\Skp}{{\mathbb S^{k}_+}\,}

\newcommand{\A}{{\mathcal A}}

\newcommand{\bbm}{\begin{bmatrix}}
\newcommand{\ebm}{\end{bmatrix}}
\newcommand{\bem}{\begin{pmatrix}}
\newcommand{\eem}{\end{pmatrix}}
\newcommand{\beq}{\begin{equation}}
\newcommand{\beqs}{\begin{equation*}}
\newcommand{\bet}{\begin{table}}
\newcommand{\eeq}{\end{equation}}
\newcommand{\eeqs}{\end{equation*}}
\newcommand{\beqr}{\begin{eqnarray}}

\DeclareMathOperator{\face}{face}
\DeclareMathOperator{\sd}{sd}
\DeclareMathOperator{\maxsd}{maxsd}
\DeclareMathOperator{\ips}{ips}

\DeclareMathOperator{\nul}{null}
\DeclareMathOperator{\range}{range}

\DeclareMathOperator{\trace}{{trace}}

\DeclareMathOperator{\Diag}{{Diag}}

\DeclareMathOperator{\relint}{{relint}}

\DeclareMathOperator{\rank}{{rank}}
\DeclareMathOperator{\interior}{{int}}
\DeclareMathOperator{\spanl}{{span}}

\DeclareMathOperator{\supp}{{supp}}

\newcommand{\nc}{\newcommand}
\nc{\arrow}{{\rm arrow\,}}
\nc{\Arrow}{{\rm Arrow\,}}
\nc{\BoDiag}{{\rm B^0Diag\,}}
\nc{\bodiag}{{\rm b^0diag\,}}

\nc{\Mm}{{\mathcal M}^{m} }
\nc{\Mmn}{{\mathcal M}^{mn} }
\nc{\Mnr}{{\mathcal M}_{nr} }
\nc{\Mnmr}{{\mathcal M}_{(n-1)r} }
\nc{\kwqqp}{Q{$^2$}P\,}
\nc{\kwqqps}{Q{$^2$}Ps}

\nc{\notinaho}{(X,S)\in \overline{AHO}(\A)}
\nc{\inaho}{(X,S)\in AHO(\A)}

\newcommand{\bea}{\begin{eqnarray}}%
\newcommand{\eea}{\end{eqnarray}}%
\newcommand{\beas}{\begin{eqnarray*}}%
\newcommand{\eeas}{\end{eqnarray*}}%
{}

\newcommand{\Hnp}[1][]{\,\mathbb{H}_+^{\ifthenelse{\equal{#1}{}}{n}{#1}}}
\newcommand{\Hn}[1][]{\,\mathbb{H}^{\ifthenelse{\equal{#1}{}}{n}{#1}}}
\newcommand{\Dn}[1][]{\,\mathbb{D}^{\ifthenelse{\equal{#1}{}}{n}{#1}}}

\begin{document}

\bibliographystyle{plain}
\title{
Implicit Redundancy and Degeneracy in Conic Program}
  \author{
Haesol Im\thanks{j5im@uwaterloo.ca} 
}
\date{ 
		\today\\
}
\maketitle




\vspace{-.4in}

\begin{abstract}
This paper examines the feasible region of a standard conic program represented as the intersection of a closed convex cone and a set of linear equalities. 
It is recently shown that when Slater constraint qualification (strict feasibility) fails for the classes of linear and semidefinite programs, two key properties emerge within the feasible region; (a) every point in the feasible region is degenerate; (b) the constraint system inherits implicit redundancies.
In this paper we show that degeneracy and implicit redundancies are inherent and universal traits of all conic programs in the absence of strict feasibility.
\end{abstract}



\section{Introduction}
\label{sec:intro}

We consider the feasible region of a standard conic program 
represented as the intersection of a closed convex cone $\cK$ and an affine subspace $\cL$. 
Throughout the paper we consider $\cK$ and $\cL$ in a finite dimensional Euclidean space~$\bE$.
We represent the affine subspace by
\[
\cL = \{ x \in \bE : \cA x = b  \in \Rm \},
\]
where $\cA : \bE \to \Rm$ is a linear map.
We define the feasible region of a conic program
\[
\cF :=  \cK \cap \cL =  \{x\in \cK: \cA x = b\} \ne \emptyset.
\]
We assume that $\cA$ is \emph{surjective}. 
Since $\cF$ is not empty, $\cA$ being surjective means that $\cA x=b$ itself does not contain any redundant equalities. 
If $\cA$ is not given surjective, numerical methods \cite{MR4054120,AndersenErlingD95,CartisGouldRelint06}
can be used for sorting out redundant equalities from the system $\cA x = b$.

$\cF$ is said to satisfy strict feasibility (Slater constraint qualification) if $\cF$ contains a point in the relative interior of $\cK$. 
This paper particularly focuses on $\cF$ that fails strict feasibility, i.e.,~$\cF$ that lies on the relative boundary of $\cK$.
It is well-known that conic programs that fail this regularity condition does not guarantee strong duality, i.e., the primal and dual optimal values may not agree or the dual optimal value is not attained by a dual feasible point; see~\cite[Section 2]{alma9941611053505162}.
For this reason, many optimization algorithms are developed under this regularity condition. 
The class of linear programs ($\cK = \Rnp$) avoids this pathology and consequences of absence of strict feasibility for linear programs are rarely discussed in the literature. 
A recent work \cite{ImWolk:22} highlights the importance of  strict feasibility for linear programs and reveals the connection between degeneracy of extreme points and strict feasibility.

The cone of $n$-by-$n$ positive semidefinite matrices, $\Snp$, is a well-studied object. The computationally efficient facial structure of $\Snp$ is clearly identified and there are successful numerical solvers that take advantage of this knowledge \cite{mosek,TohKim-Chuan2012OtIa,sturm1999using}.
Let $\cF_{\Snp} := \{x\in \Snp : \cA x = b\}$.
When $\cF_{\Snp}$ fails strict feasibility, 
two interesting properties are shown recently:
\begin{enumerate}
\item \emph{Every} point in $\cF_{\Snp}$ is \emph{degenerate} \cite{IJTWW:23}.
\item When the domain of $\cA$ is restricted to be the minimal face of $\Snp$ containing $\cF_{\Snp}$, $\cA$~is not surjective. In other words,
$\cF_{\Snp}$ contains \emph{implicitly redundant} linear constraints~\cite{HaesolIm:2022}.
\end{enumerate}
We explanations for degeneracy and implicit redundancy in \Cref{sec:Degeneracy,sec:MainResult} respectively.
In this paper we show that these two properties are \emph{universal} properties of all conic programs that fail strict feasibility, extending beyond linear and semidefinite programs. 
Specifically, for any conic program that fails strict feasibility, 
we demonstrate that every feasible point in $\cF$ is degenerate, based on the concept of degeneracy introduced in \cite{PatakiSVW:99}.
Additionally, we show that implicit redundancy, characterized by the inherent loss of surjectivity when the domain of the linear map $\cA$ is restricted to the the minimal face containing the feasible region, occurs in any conic program that fails the regularity condition.

\subsection{Notation}
\label{sec:Notation}

Given a linear map $\cT$, we let $\cT^*$ be the adjoint of $\cT$.
We use $\< \cdot, \cdot \>_{\bE}$ to denote the inner product between two vectors in the Euclidean space $\bE$; we omit the subscript $\bE$ when the meaning is clear. 
We use $\cR(\cT)$  to denote the range of the map~$\cT$. 
Given a matrix $X$, 
$\range(X)$ denotes the range of $X$, formed by the linear combinations of columns of $X$.
$\nul(X)$ is used to denote the null-space of $X$.

Given a set $\cX \subset \bE$, 
$\interior (\cX)$ ($\relint(\cX)$, resp.) means the interior (relative interior, resp.) of $\cX$. 
The dimension and the linear span of $\cX$ are denoted $\dim(\cX)$ and $\spanl (\cX)$, respectively.
We use $\cX^\perp$ to indicate the orthogonal complement of $\cX$.
Given a vector $x$, we often use $x^\perp$ to mean $\{x\}^\perp$.
$\Rn$ and $\Sn$ denote the usual vector spaces of vectors of $n$ coordinates and  $n$-by-$n$ symmetric matrices, respectively. 
Given a matrix $A$ of size $m$-by-$n$ and an index set $\cI\subseteq \{1,\ldots,n\}$, we use $A(:,\cI)$ for the submatrix of $A$ formed by the columns of $A$ that correspond to $\cI$.

$\Rnp$ is used for the nonnegative orthant $\{ x\in \Rn : x_i \ge 0, \forall i\}$ and
$\Snp$ is used for the set of $n$-by-$n$ positive semidefinite matrices, $\{ X \in \Sn :  \< y,Xy\>_{\Rn} \ge 0, \forall y \in \Rn \}$. 
When a general conic program is discussed, we omit the subscript and use the symbol~$\cF$.
We use the notation $\cF_\cK$ to distinguish 
a particular setting for the cone $\cK$ from a general conic program.
For example, when the feasible region of the semidefinite program is discussed, we use $\cF_{\Snp}$.

\subsection{Related Work}
\label{sec:RelatedWork}

The concept of implicit redundancies is introduced recently \cite{ImWolk:22,HaesolIm:2022}.
Using this notion, \cite{ImWolk:22} shows that
every extreme point (basic feasible solution) of $\cF_{\Rnp}$ is degenerate in the absence of strict feasibility.
\cite{HaesolIm:2022} shows that $\cF_{\Snp}$ also contains implicit redundant constraints in the absence of strict feasibility.
Using the notion of implicit redundancies, \cite{HaesolIm:2022} 
tightens the Barvinok-Pataki bound \cite{bar01,GPat:95}, and further improves the bound presented in \cite{ImWolk:21}.
On the degeneracy side, 
a recent work \cite{IJTWW:23} links the definition of degeneracy introduced by Pataki 
\cite[Definition 3.3.1]{PatakiSVW:99} to realize that every point in $\cF_{\Snp}$ is degenerate. This is done by exploiting the facial structure of $\Snp$.
Since a linear program is a special case of a semidefinite program, it immediately implies that every point in $\cF_{\Rnp}$ is degenerate, not just limited to the extreme points.

\section{Background}
\label{sec:Background}

This section reviews basic notions and known results in the literature related to the main results.
We present basic definitions including cones and faces, the goal of facial reduction process, and the notion of degeneracy.

\subsection{Faces of Cones}
\label{sec:Prelim}

A closed convex set $\cK \subseteq \bE$ is said to be a \textdef{cone} if 
$x\in \cK$ implies  $\alpha x \in \cK$, $\forall \alpha \ge 0 $.
The \textdef{dual cone} of~$\cK$ is $\cK^* := \{z \in \bE: \<z, x\> \ge 0, \forall x \in \cK \}$.
A convex set $f$ is called a \textdef{face} of $\cX \subseteq \bE$ (denoted $f \unlhd\cX$) if 
\[
\text{ for } x,y\in \cX \text{ with } 
\{ \lambda x + (1-\lambda) y :  \lambda \in (0,1) \} \subset f 
\implies x,y\in f.
\]
A face $f \unlhd \cK$ is said to be \textdef{exposed} if 
$f = \cK \cap \{z\}^\perp$ for some $z \in \cK^*$.
$\cK$ is said to be \textdef{facially exposed} if every face of $\cK$ is exposed. 
Given a convex set $\cX \subset \cK$, the \textdef{minimal face} of $\cK$ containing $\cX$, denoted $\face(\cX, \cK)$,  is
the intersection of faces of $\cK$ that contains $X$, i.e.,
\[
\face(\cX, \cK) 
:=  \cap \{f \unlhd \cK : \cX \subseteq f \}.
\]
It is known that a point $\hat{x} \in \relint (\cX)$ provides the characterization $\face(\cX, \cK) = \face (\hat{x}, \cK)$; see~\cite[Proposition 2.2.5]{Cheung:2013}.
Given a face $f \unlhd\cK$, the \textdef{conjugate face} of $f$, denoted $f^\Delta$, is 
\[
f^\Delta := \{ z \in \cK^* : \< z,x\> = 0, \ \forall x \in f \}.
\]

Understanding the facial structure of a cone is important for designing a computationally efficient algorithm. 
Examples of cones that demonstrate great computational powers include the nonnegative orthant $\Rnp$, the positive semidefinite cone $\Snp$, and the second-order (Lorentz) cone. Moreover, these cones lead to many successful interior point methods for solving the linear, semidefinite, second-order cone programs and the programs with Cartesian products of these cones \cite{mosek,TohKim-Chuan2012OtIa,sturm1999using}.

We list facial characterizations of two cones. 
Every face $f \unlhd \Rnp$ has the form 
\begin{equation}
\label{eq:faceRnp}
f 
= \Rnp \cap \{ x \in \Rn :x_i =0 , i\in \cI \}
= \{ V v : v \ge 0 \}
= V \R_+^{|\cI^c|} ,
\end{equation}
where $V = I_n( :, \cI^c)$ for some index set $\cI \subseteq \{1,\ldots, n\}$ and $I_n$ is the $n$-by-$n$ identity matrix. Here, $\cI^c = \{1,\ldots, n\} \setminus \cI$, the complement of $\cI$.
$\Snp$ is another example of a well-understood cone. 
Every face $f \unlhd \Snp$ has the form 
\begin{equation}
\label{eq:faceSnp}
f 
= \Snp \cap (Q Q^T)^\perp  = V \bS_+^{\rank(\hat{X})} V^T ,
\end{equation}
where $\range(Q) = \nul(\hat{X})$ for some $\hat{X} \in \relint(f)$ and $V\in \R^{n \times \rank(\hat{X})}$ satisfies $\range(V) = \range(\hat{X})$.
Both $\Rnp$ and $\Snp$ are facially exposed.
Other examples of facially exposed cones are
the second-order cone, the doubly nonnegative cone,
the hyperbolicity cone \cite[Theorem 23]{Renegar2006} and the cone of Euclidean distance matrices \cite{MR2166851}.
On the other hand, not every cone is facially exposed. 
For example, 
the exponential cone \cite{LindstromScottB.2023Ebfr},
the completely positive cone \cite{ZhangQinghong2018CpcA}, for $n\ge 5$, and the
copositive cone \cite{Dickinson:10}  are known to be not facially exposed.

\subsection{Facial Reduction for Conic Programs}
\label{sec:FRforConic}

Facial reduction (\FRp) is first introduced by Borwein and Wolkowicz in the $1980$'s~\cite{bw2,bw3}.
In-depth analysis for conic programs with linear objective functions, especially for semidefinite programs, has been conducted for the last two decades. 
Successful applications of \FR are the semidefinite relaxations of NP-hard binary combinatorial optimization problems (e.g., see \cite{KaReWoZh:94,ForbesVrisWolk:11}); and an analytical tool for measuring hardness of solving a problem numerically (e.g., see \cite{S98lmi,SWW:17}).

Facial reduction seeks to form an equivalent problem that satisfies strict feasibility. 
\FR strives to find the minimal face of $\cK$ that contains $\cF$ in order to represent the set
\[
\cF 
= 
\{x \in \cK : \cA x = b \}
=
\{x \in \face(\cF, \cK) :  \cA x = b \} .
\]
Finding $\face(\cF,\cK)$ is often a challenging task
and hence we often rely on an auxiliary lemma. 
\Cref{thm:thmAlt} below lies in the heart of the \FR algorithm.

\begin{lemma}[Theorem of the Alternative]
\label{thm:thmAlt}
For a feasible system $\cF$, exactly one of the following holds. 
\begin{enumerate}
\item 
Strict feasibility holds for $\cF$;
\item 
There exists $y\in \Rm$ such that 
\begin{equation}
\label{eq:AuxSys}
\cA^* y \in \cK^* \setminus \{\cK\}^\perp , \text{ and } \<b,y\> =0.
\end{equation}
\end{enumerate}
\end{lemma}

We note that a solution $y$ to \eqref{eq:AuxSys} yields
\[
0 = \<b,y\>_{\Rm} = \<\cA x, y \>_{\Rm} = \< x, \cA^* y\>_{\bE} , \ \forall x \in \cF.
\]
In other words, every feasible point $x\in \cF$ lies in the orthogonal complement of $\cA^* y$, i.e., 
\[
\cF \subseteq \cK \cap (\cA^* y)^\perp.
\]
Consequently, we can rewrite the set $\cF$ by 
\begin{equation}
\label{eq:FwithExposed}
\cF = \{ x \in \cK \cap (\cA^*y)^\perp  : \cA x = b \}.
\end{equation}
The ambient space now reduces from  $\bE =  \spanl (\cK)$
to $\spanl ( \cK \cap (\cA^* y)^\perp )$; the dimension reduction is an attractive by-product of the \FR process.
This completes one iteration of the \FR algorithm. 
Then the \FR algorithm replaces the cone as in \eqref{eq:FwithExposed} (i.e., $\cK \leftarrow \cK \cap (\cA^* y)^\perp$) and solves 
the system~\Cref{eq:AuxSys} repeatedly 
until a strictly feasible point with respect to the reduced cone is guaranteed. 
\Cref{algo:FRproc} summarizes the \FR process; see also \cite{LiuMinghui2018Edas,permfribergandersen,MW13}.

\begin{algorithm}
\begin{algorithmic}
\STATE Given $\cF = \{x\in \cK: \cA x  = b\}$, $\cF^0 = \cF$, $\cK^0 = \cK$, $k=0$
\WHILE{strict feasibility fails for $\cF^k$}
\STATE $k\leftarrow k+1$
\STATE Find a vector $y^k$ satisfying 
$\cA^* y^k \in (\cK^{k-1} )^* \setminus  (\cK^{k-1})^\perp$ and $\<b,y^k\> = 0$
\STATE $\cK^k \leftarrow \cK^{k-1} \cap (\cA^* y^k)^\perp$
\STATE $\cF^k \leftarrow \{ x \in \cK^k : \cA x = b \}$ 
\ENDWHILE
\STATE \textbf{Return} $\cF^k$
\end{algorithmic}
\caption{\FR Process}
\label{algo:FRproc}
\end{algorithm}

For finding solutions to \eqref{eq:AuxSys} for the cones $\Rnp$, the second-order cone and $\Snp$, open source solvers such as SDPT3 \cite{Toh01011999} and SeDuMi \cite{sturm1999using} can be used. 
However, using the interior point method for facial reduction presents certain challenges. Typically, interior point methods terminate once a near-optimal solution is found. For linear programs, an approach is presented to overcome the accuracy issue by leveraging the machine accuracy provided by the simplex method  \cite{ImWolk:22}.
For optimization over the exponential cone $K_{exp}$, Mosek \cite{mosek} provides support, along with many other cone types, including the aforementioned ones.  
As for the optimization over the doubly nonnegative, copositive and  completely positive cones, using the interior point method to solve \eqref{eq:AuxSys} is difficult due to the computationally intractable search direction step.  For problems involving the doubly nonnegative cone, solvers such as \cite{CerulliMartina2021IAfs,doi:10.1080/10556788.2019.1576176} are proposed, which follow the a variant of alternating direction method of multipliers framework.  
As for optimization over the copositive or completely positive cones, these problems are NP-hard \cite{Dur10_copositive}. While approximation methods are available, there is no guarantee that a solution to \eqref{eq:AuxSys} is found upon termination.

Upon the termination of the \FR process \Cref{algo:FRproc} assuming that accurate solutions are obtained), we have
\[
\cF \subseteq \cK \cap (\cA^* y^1)^\perp \cap  \cdots \cap (\cA^* y^{k})^\perp .
\]
Thus \FR process allows an alternative representation of $\cF$ that satisfies strict feasibility:
\[
\cF = \left\{ x \in  \cK \cap (\cA^* y^1)^\perp \cap  \cdots \cap (\cA^* y^{k})^\perp : \cA x = b \right\}.
\]
It is important to note that the \FR process does not alter the points in the set $\cF$, i.e.,
\begin{equation}
\label{eq:same_Fi}
\cF = \cF^1 = \cF^2 =\cdots \cF^k, 
\end{equation}
where each $\cF^i$ is defined in \Cref{algo:FRproc}.
We emphasize that the representation of $\cF^i$ is different from $\cF^j$, when $i\ne j$.

The shortest length of \FR process is called the \textdef{singularity degree} of $\cF$, $\sd(\cF)$, first proposed by Sturm \cite{S98lmi} for the class of semidefinite programs. This implies that \Cref{algo:FRproc} requires at least $\sd(\cF)$ number of iterations. 
Indicated by \eqref{eq:same_Fi}, we have $\cF = \cF^{\sd(\cF)}$ and we use the superscript to distinguish the different representations of the set $\cF$. 
Hence strict feasibility fails for $\cF$, $\sd(\cF) \ne \sd(\cF^{\sd(\cF)})$ and $\sd(\cF^{\sd(\cF)}) =0$.
While it is possible to construct a general conic program that has an arbitrarily large singularity degree, linear programs have $\sd(\cF_{\Rnp}) \le 1$ regardless of the problem dimension.
$\sd(\cF)$ is often used as a measure of the hardness of solving a problem numerically by engaging error bounds; see~\cite{SWW:17,S98lmi}.  


\Cref{example:exponentialCone} below is an illustration of the \FR process.
We leave more details on the \FR algorithm to
\cite{Cheung:2013,DrusWolk:16} for general conic programs,
\cite{Sremac:2019,ImWolk:22} for $\cF_{\Snp}$ and $\cF_{\Rnp}$ and related examples. 

\begin{example}[\FR on the exponential cone]
\label{example:exponentialCone}
Let  
\[
\Kexp := \{ (x,y,z): y>0, z \ge y e^{x/y}  \} \cup 
\{ (x,y,z) : x\le 0, y = 0, z \ge 0 \}
\]
be the exponential cone. Visual illustrations of $\Kexp$ can be found in \cite{LindstromScottB.2023Ebfr,mosek}.
The dual cone of $\Kexp$ is
\[
\Kexp^* = \{ (x,y,z): x<0, ez \ge - x e^{y/x}\} \cup
\{ (x,y,z) : x=0, y\ge 0, z \ge 0 \}.
\] 
With $A=\begin{bmatrix}  0 & 1 & 0 \\ 1 & -1 & 0 \end{bmatrix}$ and $b= \begin{pmatrix} 0 \\ 0 \end{pmatrix}$,
consider the following feasible set
\[
\cF_{\Kexp} := \left\{ (x,y,z) \in \Kexp : 
A (x,y,z)^T = b \right\} = \{(0,0,z) : z\ge 0 \} .
\] 
It is known that $\{(0,0,z): z\ge 0\}$ is a non-exposed extreme ray of $\Kexp$; see~\cite{LindstromScottB.2023Ebfr}.

Let $\lambda^1 = ( 1 , 0 )^T$. 
Then we have $A^* \lambda^1 = (0,1,0)^T \in \Kexp^*$ and \Cref{thm:thmAlt} yields 
that 
\[
\cF_{\Kexp}\subseteq \Kexp \cap (A^* \lambda^1)^\perp 
= \{ (x,y,z) : x\le 0, y = 0, z \ge 0  \}.
\]
Now $\lambda^2 = 
( -1 , -1 )^T $ yields 
$A^* \lambda^2 = ( -1,0,0 )^T$.
Consequently we obtain
\[
\cF_{\Kexp} \subseteq 
\Kexp \cap (A^* \lambda^1 )^\perp \cap (A^* \lambda^2 )^\perp 
= 
\{ (0,0,z) :  z \ge 0  \}.
\]

Note that 
$\lambda^1 = (\lambda^1_1,\lambda^1_2)^T$ with a nonzero second element cannot satisfy \eqref{eq:AuxSys}. 
If $\lambda^1_2 \ne 0$, $A^* \lambda^1  = (\lambda^1_2, \lambda^1_1 - \lambda^1_2, 0)^T$ must belong to  $\{ (x,y,z): x<0, ez \ge - x e^{y/x}\}$. 
However, this is impossible since the third element of  $A^* \lambda^1 $ is equal to $0$ and $\lambda^1_2 < 0$.
Thus $\lambda^1 = (1,0)^T$ is the only vector that satisfies \eqref{eq:AuxSys} up to a scalar multiple. 
We conclude that the singularity degree of this instance is $2$.
\end{example}


It is well-known that \FR applied to $\cF_{\Rnp}$ and $\cF_{\Snp}$ generates $\cF_{\Rkp}$ and $\cF_{\Skp}$, where~$k<n$. 
In plain language, \FR for linear and semidefinite programs produces another semidefinite and linear programs in smaller dimensional spaces.
This can be seen from the facial characterizations \eqref{eq:faceRnp} and \eqref{eq:faceSnp}.
\FR process does not necessarily result in a conic program in the same class, e.g., \cite[Remark 3.10]{HuHao2023Frfs}. This is also observed in  \Cref{example:exponentialCone}; $\Kexp$ is not polyhedral but 
the first \FR iteration in \Cref{example:exponentialCone} produces a polyhedral cone $\Kexp \cap (A^* \lambda^1)^\perp$.

The feasible system $\cF$ is known to hold strict feasibility generically \cite{PatTun:97} and \FR may seem unnecessary. 
However, many practical real-life examples seem to fail strict feasibility, e.g., see~\cite{ImWolk:22,Hu2022robustinteriorpoint}. 
As addressed in \cite{ImWolk:22}, a significant number of instances from the well-known NETLIB collection fails strict feasibility. 
We note that the significance of \FR on $\cF_{\Rnp}$ has only been a recent topic of discussion.

It is important to note that \FR enables the restriction of the domain within which we can operate with $\cA$.
Without \FRp, we cannot restrict the domain of $\cA$ without any further analysis. 
We utilize this observation to discuss the main results in \Cref{sec:MainResult} below.

\subsection{Degeneracy}
\label{sec:Degeneracy}

In this section we introduce a known notion of degeneracy and discuss related results.

\begin{definition}\cite[Definition 3.3.1]{PatakiSVW:99}
\label{def:degen}
Let $\bar{x}\in \cF$. We say that $\bar{x}$ is \textdef{nondegenerate} if 
\[  
\spanl \face(\bar{x}, \cK)^\Delta \cap \cR(\cA^*) = \{0\} .
\]
\end{definition}

Characterizations of \Cref{def:degen} are available for some cones. 
Given $\bar{x} \in \cF_{\Rnp}$, 
let $\supp(\bar{x})$ be the support of $\bar{x}$, $\{i \in \{1,\ldots, n\} : \bar{x}_i \ne 0\}$.
Let $A$ be a matrix representation of $\cA: \Rn \to \Rm$. 
Then $\bar{x}$ is called degenerate if $A(:,\supp(\bar{x}))$ has linearly dependent rows; see~\cite[Example 3.3.1]{PatakiSVW:99}.
We note that the usual degeneracy of a basic feasible solution for the standard linear program agrees with the degeneracy discussed in \cite[Example 3.3.1]{PatakiSVW:99}.
A basic feasible solution $x^\prime$ of the standard linear program is called degenerate 
if there is a basic variable equal to $0$; see \cite[Section 2]{BT97}.
This implies that 
$A (:,\supp(x^\prime))$ (a submatrix of the so-called basis matrix) is a full-column rank matrix with more rows than columns. 
The characterization immediately indicates that $x^\prime$ is degenerate.

A nice characterization of \Cref{def:degen} customized for  $\cF_{\Snp}$ is available as well.
For $\bar{X} \in \cF_{\Snp}$, let 
$\bar{X} = 
\begin{bmatrix} V_1 & V_2  \end{bmatrix} 
\begin{bmatrix} R & 0 \\ 0 & 0 \end{bmatrix}
\begin{bmatrix} V_1^T \\ V_2^T  \end{bmatrix}$ be a spectral decomposition of $\bar{X}$, where $R$ is a positive definite matrix of order $\rank(\bar{X})$.
Let $\<A_i, X\>_{\Sn} = b_i$ be the $i$-th equality of the system $\cA X =b \in \Rm$.
Then \cite[Corollary 3.3.2]{PatakiSVW:99} states 
\begin{equation}
\label{eq:rotate_Ai}
\text{$\bar{X}$ is degenerate $\iff$ } 
\left\{ \begin{bmatrix}
V_1^T A_i V_1 & V_1^T A_i V_2  
\end{bmatrix}  \right\}_{i=1}^m 
\text{ contains linearly dependent matrices.}
\end{equation}
In \cite[Section 2]{IJTWW:23}, the degeneracy of each point in $\cF_{\Snp}$ is realized using the characterization~\eqref{eq:rotate_Ai}.
We briefly explain the steps for recognizing the degeneracy. 
Since $\Snp$ is facially exposed, 
we always obtain 
$\face(\cF,\Snp) = V \bS_{+}^r V^T$ for some full column rank matrix $V \in \R^{n\times r}$.  
Then a solution $y$ to \eqref{eq:AuxSys} is used to detect the linear dependence of the matrices in \eqref{eq:rotate_Ai} after \FRp.
We leave more details to \cite[Section 2]{IJTWW:23}.
The characterizations of degeneracies in the case of $\cF_{\Rnp}$ and $\cF_{\Snp}$ owe to the computationally amenable facial structures \eqref{eq:faceRnp} and \eqref{eq:faceSnp}.
This again emphasizes the importance of understanding facial geometry of cones.

\section{Main Result}
\label{sec:MainResult}

We now present the main results of this paper.

\begin{theorem}[Degeneracy in the absence of strict feasibility]
\label{thm:degenAll}
Suppose that strict feasibility fails for $\cF$.
Then every point in $\cF$ is degenerate. 
\end{theorem}

\begin{proof}
Suppose that strict feasibility fails for $\cF$.
Then by \Cref{thm:thmAlt}, there exists $y\in \Rm$ that solves \eqref{eq:AuxSys}.
Note that $\cA^*y \in \cR(\cA^*)$ and $\cA^*y \ne 0$.
Now let $\bar{x} \in \cF$. 
Then 
\[
\<\cA^*y,\bar{x}\> = \<y,\cA \bar{x} \> = \<y,b\> =0 \text{ and } \cA^* y \in \cK^*.
\]
Hence $\cA^*y \in \face(\bar{x}, \cK)^\Delta$.
Therefore we conclude that 
$\cA^* y  \in \spanl \face(\bar{x}, \cK)^\Delta $.
Consequently, by \Cref{def:degen}, $\bar{x}$ is degenerate.
\end{proof}

Now the results in \cite{ImWolk:22,IJTWW:23} can be viewed as a consequence of \Cref{thm:degenAll}.
An immediate application of \Cref{thm:degenAll} yields that
$\cF$ that has a nondegenerate point holds strict feasibility. 
Moreover, $\cF$ that holds strict feasibility always has a nondegenerate point as we see in \Cref{prop:slater_Nondegn} below.

\begin{prop}
\label{prop:slater_Nondegn}
Every strictly feasible point to $\cF$ is nondegenerate.
\end{prop}
\begin{proof}
Let $\bar{x}$ $\in \cF$ be a strictly feasible point. 
Then $\face(\bar{x}, \cK) = \cK$. 
If $\interior \cK \ne \emptyset$, then \Cref{def:degen} immediately follows. 
Suppose to the contrary that
\[
0 \ne \phi \in \spanl \  \cK^\Delta \cap \cR(\cA^*)  .
\]
Then $\phi = \cA^* z$ for some $z$. 
Note that $\cA^*z \in \spanl \ \cK^\Delta$ and hence $\<\cA^*z, x\> = 0$, for all $x\in \cK$. 
This yields  $\cA^*z \in \cK^*$ and $0 = \<\cA^*z, \bar{x}\> = \<z,\cA \bar{x}\> = \<z,b\>$.
Therefore $z$ satisfies~\eqref{eq:AuxSys} and this contradicts the strict feasibility assumption.
\end{proof}

\Cref{thm:degenAll}  and \Cref{prop:slater_Nondegn} lead to the following conclusion.
\begin{corollary} Given any feasible set $\cF$, 
strict feasibility fails if, and only if,
every point is degenerate. 
Moreover, strict feasibility holds if, and only if, $\cF$ contains a nondegenerate point.
\end{corollary}
If we restrict our attention to the extreme points of $\cF$ only, the converse of \Cref{thm:degenAll} is not true. 
For example, $\cF_{\Rnp}$ that represents the set of doubly stochastic matrices (the feasible region of the linear assignment problem) has a Slater point, but every extreme point is degenerate; see~\cite{ImWolk:22}.

We now proceed to the second main result. 
We first introduce a useful lemma.
We include the proof of \Cref{lem:bexposed} for completeness. 
\begin{lemma}\cite[Lemma 3.6]{WangWolk:22}
\label{lem:bexposed}
Suppose that strict feasibility fails for $\cF$. 
Let $y$ be a solution to~\eqref{eq:AuxSys}.
Then 
\begin{equation}
\label{eq:yisexposing}
\cA (\cK \cap (\cA^*y)^\perp ) = \cA(\cK) \cap y^\perp.
\end{equation}
\end{lemma}

\begin{proof}
Let $w\in \cA(\cK \cap (\cA^*y)^\perp)$. Then $w = \cA x $ for some $x\in \cK \cap (\cA^*y)^\perp$. Clearly $w\in \cA (\cK)$.
Note that $\< w,y \> = \< \cA x, y\> = \<x,\cA^* y\> =0$. Hence $\cA (\cK \cap (\cA^*y)^\perp ) \subseteq \cA(\cK) \cap \{y\}^\perp$ holds.
For the opposite containment, let $w\in \cA(\cK) \cap \{y\}^\perp$.
Then $w = \cA x$ for some $x\in \cK$. 
Furthermore, $0 = \<w,y\> = \< \cA x, y\> = \< x, \cA^* y\> $.
\end{proof}

Recall that $\cA$ is assumed to be surjective throughout this paper. 
We emphasize that $\cA x = b$ alone does not have redundant equalities.
\Cref{thm:implicit_red} below shows that
$\cA$ is no longer surjective 
when the domain of $\cA$ is restricted by the orthogonal complement of $\cA^* y$, where $y$ is a solution to \eqref{eq:AuxSys}.
That is, the constraint system $\cF$ contains \textdef{implicit redundant} constraints. 

\begin{theorem}[Implicit loss of surjectivity in the absence of strict feasibility]
\label{thm:implicit_red}
Suppose strict feasibility fails for $\cF$. 
Let $\{y^i\}_{i=1}^{k}$ be the set of vectors obtained by the \FR process.
Let $\bar{\cK} =  \cK \cap_{i=1}^{k} (\cA^* y^i )^\perp$ and let $\bar{\bE} = \spanl(\bar{\cK})$.
Then the map $\bar{\cA} : \bar{\bE} \to \Rm$ is not surjective. 
In other words, the equality constraint system in 
$\{x\in \bar{\cK} : \cA x = b \}$ contains redundant constraints. 
\end{theorem}
\begin{proof}
Suppose strict feasibility fails for $\cF$.
Let $\bar{\cK}$ be the closed convex cone obtained by the \FR process, i.e., 
$\bar{\cK} := \cK \cap_{i=1}^k (\cA^* y^i)^\perp$, where $\{y^i\}_{i=1}^k$ is obtained by solving \eqref{eq:AuxSys}.
Then by \Cref{lem:bexposed}, $y^1\in \Rm$ satisfies 
\eqref{eq:yisexposing}. 
Thus the containment below follows:
\[
\cA (\bar{\cK}) \subseteq \cA (\cK \cap (\cA^*y^1)^\perp ) = \cA(\cK) \cap (y^1)^\perp .
\]
Note that $y^1$ is nonzero since $\cA^* y^1 \ne 0$.
Hence the range of $\bar{\cA}$ cannot span~$\Rm$.
\end{proof}

\Cref{thm:implicit_red} leads to the notion of \textdef{implicit problem singularity}, $\ips$,
discussed \cite{ImWolk:22,HaesolIm:2022,IJTWW:23}.
We state the definition of $\ips$ for a general conic program.
\begin{definition}
\label{def:ips}
Let $\cK \cap (\cA^*y^1)^\perp \cap \cdots \cap ( \cA^*y^{k})^\perp $ be the cone obtained by the \FR process.
The implicit problem singularity of $\cF$, $\ips(\cF)$, is the number of  redundant equalities in the system 
\[
\left\{ x \in \cK \cap (\cA^*y^1)^\perp \cap \cdots \cap ( \cA^*y^{k})^\perp  : \cA x = b \right\}.
\]
\end{definition}
We discuss notions related to \Cref{def:ips} and  their usages in \Cref{sec:Discussions} below.

\begin{remark}
The implicit redundancies provide an interesting implication to 
the minimal representation of the affine space discussed in \cite{ScTuWominimal:07}.
Let $\cN(\cA)$  be the null-space of $\cA$. Then, given~$\hat{x} \in \cL = \{x \in \bE : \cA x = b\}$, we obtain
$\cL = \hat{x} + \cN(\cA) $. 
Since $\cK \cap \cL = \face(\cF, \cK) \cap \cL$ and  $(\face(\cF,\cK)-\face(\cF,\cK)) \cap \cK = \face(\cF,\cK)$,
 \cite{ScTuWominimal:07} identifies the equality   
\[
\cF 
= ( \hat{x} +\cN(\cA) ) \cap \cK  
= ( \hat{x} + \cN(\cA)) \cap (\face(\cF,\cK)-\face(\cF,\cK)) \cap \cK .
\]
Then \cite[equation (2.18)]{ScTuWominimal:07} 
establishes $\cL_{DM} = ( \hat{x} + \cN(\cA)) \cap (\face(\cF,\cK)-\face(\cF,\cK))$ as `the minimal subspace representation' of $\cL$.
\Cref{thm:implicit_red}  and  \Cref{def:ips}
certify that~$\cL_{DM}$ indeed has a smaller dimension than $\cL$ when strict feasibility fails for $\cF$.
\end{remark}

\section{Discussions}
\label{sec:Discussions}

In this section we connect the main results to some topics in the literature, including the 
notions of different singularities and their usages, ill-conditioning that arises in the interior point methods,
and differentiability of the optimal value function.

\subsection{Other Singularity Notions}

\Cref{thm:implicit_red} indicates that the image under $\cA$ with its domain restricted by the exposing vectors from the \FR process cannot span $\Rm$. 
Yet, it does not specify how different the dimension of the image is from $\Rm$ in the absence of strict feasibility. 
We aim to provide a bound for the lost dimension.

As a related definition to the singularity degree,
a notion of \textdef{max-singularity degree} is proposed in \cite{ImWolk:22}.
The max-singularity degree of $\cF$, $\maxsd(\cF)$, is the longest nontrivial \FR iterations for $\cF$. 
We discuss some properties related to 
these singularity notions.

\begin{proposition}\cite[Theorem 1]{LiuMinghui2018Edas}
\label{prop:yi_linindep}
Suppose that $\cF$ fails strict feasibility.
Let $\{y^i\}_{i=1}^k$ generated by solving~\eqref{eq:AuxSys}.
Then the vectors in $\{y^i\}_{i=1}^k$ are linearly independent.
\end{proposition}

Note that the property stated in \Cref{prop:yi_linindep} holds for any \FR process. 
Thus \Cref{coro:maxsdIps} below follows.

\begin{corollary}
\label{coro:maxsdIps}
Suppose that strict feasibility fails for $\cF$ and let $\bar{\cK}$ be the cone obtained by the \FR process. Then
$\cA(\spanl \bar{\cK} )$ generates at most $m-\maxsd(\cF)$ dimensional space. 
Moreover, $\{ x\in \bar{\cK} : \cA x = b \}$ contains at least $\maxsd(\cF)$ number of implicitly redundant constraints. 
\end{corollary}
\begin{proof}
Let $\{y^i\}_{i=1}^{\maxsd(\cF)}$ be the sequence of vectors obtained by the \FR iterations by solving~\eqref{eq:AuxSys}.
Then \Cref{lem:bexposed} implies that
\[
\cA (\bar{\cK} ) 
= \cA \left( 
 \cK \cap (\cA^* y^1)^\perp \cap  \cdots (\cA^* y^{\maxsd(\cF)})^\perp 
 \right) = \cA(\cK) \cap \{y^1\}^\perp \cap \cdots \{y^{\maxsd(\cF)} \}^\perp.
\]
Finally, \Cref{prop:yi_linindep}
implies that $\cA(\spanl\bar{\cK})$ spans at most $m-\maxsd(\cF)$ dimensional space, since $\{y^i\}$ are linearly independent by \Cref{prop:yi_linindep}.
\end{proof}

\Cref{coro:maxsdIps} gives rise to the following relationship:
\[
\ips(\cF) \ge \maxsd (\cF) \ge \sd(\cF).
\]
This motivates the following question: how different can the $\ips, \maxsd$, and $\sd$ be?
Various instances are presented for $\cF_{\Rnp}$ and $\cF_{\Snp}$ in the literature;
$\sd(\cF_{\Snp}) = \maxsd(\cF_{\Snp}) < \ips(\cF_{\Snp})$ \cite[Example 3.2.13]{HaesolIm:2022};
$\sd(\cF_{\Snp}) = \maxsd(\cF_{\Snp}) = \ips(\cF_{\Snp})$ \cite[Example 4.2.6]{Sremac:2019}; 
and $\sd(\cF_{\Rnp}) =1 < \ips (\cF_{\Rnp}) = 275$ \cite[Section 4.2.2]{ImWolk:22}.
\Cref{example:exponentialCone} above shows that
$\sd(\cF_{\Kexp}) = \maxsd(\cF_{\Kexp}) = \ips(\cF_{\Kexp}) = 2$.

The implicit redundancies can be used to tighten bounds that engage the number of linear constraints.
An interesting usage of $\ips$ is shown for $\cF_{\Rnp}$ and $\cF_{\Snp}$ in the literature.
For the case of $\cF_{\Rnp}$, $\ips$ is used for measuring the degree of degeneracy of a basic feasible solution \cite[Corollary 4.1.12]{HaesolIm:2022}.
More specifically, every basic feasible solution has at most $m-\ips(\cF_{\Rnp})$ number of nonzero basic variables.  
This also translates to 
\begin{equation}
\label{eq:bfsips}
\text{
every basic feasible solution has at least $\ips(\cF_{\Rnp})$ degenerate basic variables. 
}
\end{equation}
In the case of $\cF_{\Snp}$, $\ips$ is used to tighten the Barvinok-Pataki bound \cite[Theorem 3.2.7]{HaesolIm:2022}:
\begin{equation}
\label{eq:extSnpIPS}
\text{every extreme point  } X\in \cF_{\Snp} \text{ satisfies }
\rank(X) (\rank(X)+1)/2 \le m - \ips(\cF_{\Snp}) .
\end{equation}
We observe that a large $\ips(\cF_{\Snp})$ provides an especially useful bound since the term on the left-hand-side of the inequality \eqref{eq:extSnpIPS} is quadratic in $\rank(X)$.

We now aim to generalize \eqref{eq:bfsips} and \eqref{eq:extSnpIPS} to an arbitrary conic program.
We first present a lemma. 
Note that \Cref{lemma:dimface_dimYrelation} can be used to derive the well-known Barvinok-Pataki bound \cite{bar01,GPat:95}. 
\begin{lemma}\cite[Theorem 3.3.1]{PatakiSVW:99}
\label{lemma:dimface_dimYrelation}
Let $\bar{x} \in \cF$.  Then
\[ 
\dim [ \face (\bar{x}, \cF)  ]
= \dim [ \face(\bar{x}, \cK) ] - m + \dim [ \face(\bar{x}, \cK)^\perp \cap \cR(\cA^*) ] .
\]
\end{lemma}

\begin{theorem}
\label{thm:tighenUsingIPS}
Let $\bar{x}$ be an extreme point of $\cF$. Then 
\[
\dim \face(\bar{x}, \cK) 
\le m - \ips(\cF) .
\]
\end{theorem}

\begin{proof}
Let $\{y^i\}_{i=1}^k$ be a set of vectors obtained by the \FR iterations.
Let 
\[
\bar{\cK} := \cK \cap (\cA^* y^1)^\perp \cap \cdots \cap (\cA^* y^{k})^\perp.
\]
Let $\cN(\cA)$ be the null-space of $\cA$.
Note that 
\begin{equation}
\label{Adimbreakup}
m = \dim \left(\cA \left( \bar{\cK} \cup \bar{\cK}^\perp \right) \right)
\le \dim (\cA \left( \bar{\cK}) \right) +  \dim  \left( \cA ( \bar{\cK}^\perp) \right) .
\end{equation}
Since  
\[ 
\begin{array}{rcl}
\cA [ \bar{\cK}^\perp ]
& = & 
\cA \left[  \left( \bar{\cK}^\perp \cap \cR(\cA^*)\right) \cup \cN(\cA)  \right] \\
& \subseteq & \cA  \left[ \bar{\cK}^\perp \cap \cR(\cA^*)\right] +  \cA\left[ \cN(\cA) \right] \\
& = & \cA  \left[ \bar{\cK}^\perp \cap \cR(\cA^*)\right],
\end{array}
\]
we have 
\begin{equation}
\label{eq:proof_aid1}
\begin{array}{rcl}
\dim [\cA ( \bar{\cK}^\perp) ]  & \le & 
\dim [ \cA  ( \bar{\cK}^\perp \cap \cR(\cA^*) ) ] \\
&   \le &
\dim [ \bar{\cK}^\perp \cap \cR(\cA^*)  ]  \\
& \le&
\dim [ \face(\cF,\cK)^\perp \cap \cR(\cA^*)  ] .
\end{array}
\end{equation}
Since $\bar{x}$ is an extreme point of $\cF$, $\dim [\face (\bar{x}, \cF)] =0$.
Then \Cref{lemma:dimface_dimYrelation} implies
\begin{equation}
\label{eq:dimfacexK_almostubd}
\dim [\face(\bar{x}, \cK) ]
=   m - \dim [\face(\bar{x}, \cK)^\perp \cap \cR(\cA^*) ] .
\end{equation}
Note that
\[
\begin{array}{rclll}
&\face(\cF,\cK)^\perp &\subseteq& \face(\bar{x}, \cK)^\perp   \\
\implies& 
\dim \left[ \face(\cF,\cK)^\perp  \cap \cR(\cA^*)\right] & \le & 
\dim \left[ \face(\bar{x}, \cK)^\perp  \cap \cR(\cA^*) \right] .
\end{array}
\]
Then we obtain
\[
\begin{array}{rcll}
\dim [\face(\bar{x}, \cK) ]
&\le &  m - \dim \left[ \face(\cF,\cK)^\perp  \cap \cR(\cA^*)\right]  &  \text{by \eqref{eq:dimfacexK_almostubd}} \\
&\le& m - \dim [ \cA ( \bar{\cK}^\perp)] & \text{by \eqref{eq:proof_aid1}}\\
& \le& \dim [ \cA(\bar{\cK} )] & \text{by \eqref{Adimbreakup}} \\
& \le & m- \ips(\cF).
\end{array}
\]
\end{proof}

\Cref{thm:tighenUsingIPS} can be used for cones for which facial dimensions are known. 
For example, $\dim \face(\bar{x}, \Rnp) = |\supp(\bar{x})|$.
Thus \eqref{eq:bfsips} can be viewed as a consequence of \Cref{thm:tighenUsingIPS}.
It is well known that 
$\dim \face(\bar{X}, \Snp) = \rank(\bar{X})(\rank(\bar{X})+1)/2$ and
\eqref{eq:extSnpIPS} also follows from \Cref{thm:tighenUsingIPS}.
Let $\COPn:= \{ X \in \Sn: \<y,Xy\>_{\Rn}, \forall y\in \Rnp \}$ be the copositive cone.
The dimensions of the so-called maximal faces of the copositive cone are identified. 
A maximal face $f$ of $\COPn$ is characterized by 
$f = \{ X \in \COPn : v^T X v =0 \}$ for some $v \in \Rnp\setminus \{0\}$; see~\cite[Theorem 5.5]{Dickinson:10}.
Moreover, $\dim f = \frac{1}{2}n(n+1)- |\supp(v)|$.
Then by \Cref{thm:tighenUsingIPS}, we conclude that 
an extreme point $\bar{X}$ of $\cF_{\COPn}$, where $\face(\bar{X}, \COPn)$ is maximal, satisfies
\[
\frac{1}{2}n(n+1) - |\supp(v)| \le m-\ips(\cF).
\]

\subsection{Interior Point Methods}

Interior point methods are the most popular class of algorithms for solving a linear conic program
\begin{equation}
\label{eq:standardPrimal}
\min_x \{ \<c,x\> : \cA x = b, x \in \cK  \ \} .
\end{equation}
A general framework of the interior point method 
applicable to the standard conic program is discussed in \cite{MR2436012,permfribergandersen}. 
We aim to identify a repercussion that arises when either the primal path-following interior point method or the extended dual embedding is applied to an instance that fails strict feasibility.
Throughout the interior point method, solving a series of linear systems is essential.
In practice, it is observed that these linear systems often become ill-conditioned as the iterates approach to an optimal point. 
We attribute this phenomenon to degeneracy arising from the lack of strict feasibility in semidefinite programming.

\paragraph{Limiting Behaviour of the Self-Concordant Barrier based Interior Point Method}
We follow the steps presented in \cite{MR2436012}. 
First we construct the so-called 
logarithmically homogeneous self-concordant barrier (LHSCB) $F: \interior(\cK) \to \R$ for the cone $\cK$. 
For $t >0$, consider 
\[
(PB_t) \quad \min_x \{ \  t \<c,x\> + F(x) : Ax = b, x \in \cK \ \} .
\]
The barrier function $F$ maintains the variable $x$ in $\interior (\cK)$. 
The barrier functions 
$F(x) = -\sum_j \ln x_j$,  
and $F(X) = -\ln \det X$ are presented as examples 
for linear and semidefinite programs, respectively.
The first-order optimality conditions for $(PB_t)$ are
\begin{equation}
\label{eq:iptKKT}
\begin{array}{l}
\cA^* y + s = c \\
\cA x = b \\
\nabla F (x) + ts = 0 , x\in \interior (\cK), s \in \interior(\cK^*)
\end{array}
\end{equation}
The Newton step for solving the optimality conditions \eqref{eq:iptKKT} can be computed by solving 
\begin{equation}
\label{eq:iptKKT2}
\begin{array}{l}
\cA^* \Delta y + \Delta s = 0 \\
\cA \Delta x = 0 \\
\frac{1}{t} \nabla^2 F (x) \Delta x + \Delta s = -s - \frac{1}{t} \nabla F(x)
\end{array}
\end{equation}
Then by substituting $\Delta s, \Delta x$ into the middle block in \eqref{eq:iptKKT2}, we obtain
\cite[equation 3.9]{MR2436012}:
\begin{equation}
\label{ADAT_ipt}
(\cA [ \nabla F^2(x) ]^{-1} \cA^* ) \overline{ \Delta y }  = \cA [ \nabla^2 F(x)]^{-1} (s + t_+^{-1} \nabla F(x) ).
\end{equation}
Note that for infeasible-interior point method, 
some appropriate residuals are added to the right-hand-side of \eqref{ADAT_ipt} but the data on the left-hand-side remains the same.

We now observe the limiting behaviour of the left-hand-side matrix $A[ \nabla F^2(x) ]^{-1} A^T$ in~\eqref{ADAT_ipt} with the barrier 
functions introduced in \cite{MR2436012}.
For the case of the linear program, 
we see that 
$[\nabla^2 F(x)]^{-1} = \Diag(x)^2 $, where $\Diag(x)$ is the diagonal matrix with $x$ placed on the diagonal elements.
Let $x_k$ and the $k$-th iterate.
Let  $x^*$ be an optimal solution to 
$(PB_t)$ and let $A_{x*} = A \Diag(x^*)$.
We note that  $A \Diag(x^*)^2 A^T = A_{x*}A_{x*}^T$ and 
$\range ( A_{x*} ) \subsetneq \Rm$.
Thus, $A \Diag(x^*)^2 A^T$ is a singular matrix. 
As $x_k \to x^*$, ill-conditioning arises.
For the case of the semidefinite program, 
the~$(i,j)$-th entry of the data matrix in~\eqref{ADAT_ipt} is $\<A_i, X A_j X \> $.
Let $X^*$ be an optimal solution and let $M_{X^*}$ be the matrix 
where $(M_{X^*})_{i,j} = \<A_i, X^* A_j X^* \>$. Let $y$ be a vector that solves~\eqref{eq:AuxSys}. 
Then we observe
\[
M_{X^*}y =  
\begin{bmatrix}
\sum_j y_j \trace (A_1X^*A_jX^*) \\
\vdots \\
\sum_j y_j \trace (A_mX^*A_jX^*) 
\end{bmatrix}
= 
\begin{bmatrix}
 \trace \left(A_1 X^* \sum_j y_j A_j X^* \right) \\
\vdots \\
\trace \left( A_m X^* \sum_j y_j A_j X^* \right) 
\end{bmatrix} = 0 ,
\]
since $(\sum_j y_j A_j) X^* = 0$.
Hence $M_{X^*} $ is a singular matrix. Similarly, ill-conditioning arises in \eqref{ADAT_ipt} as iterate $X_k$ approaches to $X^*$.

\paragraph{Limiting Behaviour of the Extended Dual Embedding}
We begin by introducing the extended dual embedding problem from \cite{permfribergandersen}.
Let $r_p=\cA \hat{x} -b\hat{\tau}$, 
$r_d=-\cA^* \hat{y}- \hat{s} + c\hat{\tau}$,  
$r_g = \<b,\hat{y}\> - \<c,\hat{x}\> - \hat{\kappa}$, and $\alpha = \<\hat{s}, \hat{x} \> + \hat{\tau} \hat{\kappa}$, where $\hat{y} \in \Rm$, positive definite $\hat{X},\hat{S}$ and positive scalars $\hat{\tau} , \hat{\kappa}$.
Then the extended dual embedding is  
\begin{equation}
\label{eq:selfDualEmbedding}
\begin{array}{rl}
\min & \alpha \theta \\
\text{subject to} 
& \cA x -b \tau = r_p \theta \\
& - \cA^* y -s + c \tau = r_d \theta \\
& \<b,y\> - \<c,x\> - \kappa  = r_g \theta \\
& \< r_p,y\> + \<r_d, x\> + r_g \tau= -\alpha \\
&(x,s,y,\tau, \kappa, \theta ) \in \cK \times \cK^* \times \Rm \times \R_+ \times \R_+ \times \R .
\end{array}
\end{equation}
The feasible region of \eqref{eq:selfDualEmbedding} contains an analytic strictly feasible point, hence it enables the use of the feasible interior point methods and ensures strong duality.
Let $(x^*,y^*,s^*,\tau^*, \kappa^*, \theta^*)$ be an optimal solution to \eqref{eq:selfDualEmbedding}, that is a limit point of the interior point method.  
Based on the complementarity $\tau^* \kappa^*=0$ at an optimal point, one can determine the optimality, feasibility of the original problem, or identify exposing vectors for the primal~\eqref{eq:standardPrimal} or its dual; see \cite{permfribergandersen}. We focus on the case $\tau^*>0$.

In the case of linear programs, the limiting behaviour is discussed in \cite{ImWolk:22}.
We focus on a semidefinte programming.
Solving \eqref{eq:selfDualEmbedding} incorporates the 
perturbed complementarities $XS =\mu I$, $\tau\kappa = \mu$, $\mu >0$ to the constraints and we solve the system \eqref{eq:subeqSD}  to obtain the Newton direction:
\begin{subequations}
\label{eq:subeqSD}
\begin{align}
& \quad \ \  \cA \Delta X & -b \Delta\tau & - r_p \Delta\theta  &&&=&0 \label{eq:SDpf1}\\
- \cA^* \Delta y &&+  c \Delta \tau & - r_d \Delta \theta &-\Delta S & &=& 0 \label{eq:SDdf1} \\
\<b, \Delta y\>& - \<c,\Delta X\>&    &- r_g \Delta\theta & & -\Delta\kappa &=&0 \label{eq:SDopt1}\\
\< r_p,\Delta y\> &+ \<r_d, \Delta X\> & + r_g \Delta \tau && &&=& 0 \label{eq:SDdist} \\
&  \quad \Delta X \cdot S & & & + X\cdot \Delta S & &=&\mu I- XS \label{eq:SDcs1} \\
&&&  \kappa / \tau \Delta \tau & & +\Delta \kappa &=& \mu/\tau - \kappa  \label{eq:SDcs2} . 
\end{align}
\end{subequations}


Block variable elimination is a common strategy used to reduce the size of the linear system that needs to be solved. 
Substituting $\Delta S =-\cA^* \Delta y + c\Delta \tau - r_d \Delta \theta$ \eqref{eq:SDopt1}  into 
the equality obtained from \eqref{eq:SDcs1}, we get
\[
\Delta X = X \cA^* \Delta y  S^{-1}- X c\Delta \tau S^{-1} + X r_d \Delta \theta S^{-1} + \mu S^{-1} -X.
\]
Substituting $\Delta X$  into  \eqref{eq:SDpf1} yields
\[
\cA \left(  X \cA^* \Delta y   S^{-1} \right)
- \left( \cA \left( X c  X S^{-1} \right) + b \right) \Delta \tau
+ \left( \cA \left( X r_d  S^{-1} -r_p\right) \right)\Delta \theta = R,
\]
where $R$ is a residual, independent of  $\Delta y, \Delta \tau, \Delta \theta$.
We focus on the term associated with~$\Delta y$:
\[
\cA \left(  X \cA^* \Delta y   S^{-1} \right)
=
\sum_{i=1}^m  \Delta y_i \cA \left(  X A_i S^{-1} \right).
\]
The matrix representation the above has the $i$-th column as
$\cA \left( X A_i S^{-1} \right)$.
Next, we observe the limiting behaviour of this matrix, as $X$ approaches to an optimal $X^*$.
Let $\lambda$ be a solution to \eqref{eq:AuxSys}. 
We note that $\frac{1}{\tau^*}X^*$ is an optimal solution to the original primal problem and thus
$(\cA^*\lambda)X^*=0$. 
Therefore, 
\[
\sum_{i=1}^m \lambda_i \cA \left( X^* A_i (S^*)^{-1} \right) 
=
 \cA \left( X^* (\cA^* \lambda )  (S^*)^{-1} \right) 
=
 \cA \left(  0 \cdot  (S^*)^{-1} \right) =0.
\]
That is, the matrix is singular.


\subsection{Differentiability of the Optimal Value Function}

The differentiability of the optimal value function is an interesting topic
in the field of perturbation analysis.
For the class of linear programs, many results are applicable under the nondegeneracy setting \cite{MR820987,Fia:83}.
Let $f: \bE \to \R$ be a convex function.
It is shown in \cite{HaesolIm:2022,ImWolk:22}
that, under the setting $\cK = \Rnp$ and $\cK = \Snp$,  the optimal value function $p^* : \Rm \to \R \cup \{\infty\}$ defined by
\begin{equation}
\label{eq:optiValFun}
p^*(\Delta b) := \min_{x} \{ f(x) : \cA x = b + \Delta b , x \in \cK \}
\end{equation}
is not differentiable at $0$ in the absence of strict feasibility. 
This property still holds for an arbitrary conic programs. 
Let $\bar{y}$ be a solution to \eqref{eq:AuxSys}. 
We consider $\Delta b = - \epsilon \bar{y}$ for some $\epsilon >0$.
Then replacing the right-hand-side vector $b$ of $\cF$ by $b-\epsilon\bar{y}$ implies 
\[
\cA^* \bar{y} \in \cK^*, \text{ and } \<b-\epsilon \bar{y} , \bar{y}\> = 0 - \epsilon \| \bar{y} \|^2 <0.
\]
By Farkas' lemma, the perturbed system is infeasible. 
That is, 
if any feasible $\cF$ is perturbed~$b$ by a solution $\bar{y}$ of \eqref{eq:AuxSys}, $\cF$ become infeasible.

\section{Conclusions}

We showed two universal properties of 
the standard conic program that fails strict feasibility.
First we showed that every point is degenerate in the absence of strict feasibility. 
We also noted that the application of \FR can alter the algebraic interpretation of a point. 
For instance, in the absence of strict feasibility,
a strictly feasible point in the facially reduced set is nondegenerate, whereas the same point is degenerate before \FRp. 
Understanding the points for which \FR influences the degeneracy status presents an interesting avenue for exploration.
Second, we showed that the constraint system that fails strict feasibility inherits implicitly redundant constraints. 
We saw different usages of the singularity notions
conveyed by the \FR process.
While  $\sd(\cF)$ quantifies the computational challenge associated with solving a given problem, $\maxsd(\cF)$ provides a measure for the number of implicitly redundant linear constraints. 
We further explained on how the implicit redundancies impact the bound that engages the number of constraints and the near optimal behaviour of the interior point methods.

\section*{Acknowledgement}
The author would like to thank professor Henry Wolkowicz and the anonymous reviewers for providing helpful comments.








\cleardoublepage
\phantomsection
\addcontentsline{toc}{section}{Bibliography}
\bibliography{bib_Degeneracy_Regularity}




\end{document}